\documentclass{amsart}
\usepackage{amsmath,amssymb, amsthm}
\newtheorem{theorem}{Theorem}[section]
\newtheorem{lemma}[theorem]{Lemma}
\newtheorem{proposition}[theorem]{Proposition}
\newtheorem{corollary}[theorem]{Corollary}

\newtheorem{example}[theorem]{Example}
\newtheorem{question}[theorem]{Question}

\begin{document}

\title{Almost clean rings and arithmetical rings}
%\headlinetitle{}

\author{Fran\c{c}ois Couchot}
\address{Laboratoire de Math\'ematiques Nicolas Oresme, CNRS UMR
  6139,
D\'epartement de math\'ematiques et m\'ecanique,
14032 Caen cedex, France}
\email{couchot@math.unicaen.fr} 

%\authortwo{}
%\addresstwo{}
%\countrytwo{}
%\emailtwo{}

%\authorthree{}
%\addressthree{}
%\countrythree{}
%\emailthree{}

%\authorfour{}
%\addressfour{}
%\countryfour{}
%\emailfour{}

%\headlineauthor{}

%\researchsupported{}

\begin{abstract}
It is shown that a commutative B\'{e}zout ring $R$ with compact minimal prime spectrum is an elementary divisor ring if and only if so is $R/L$ for each minimal prime ideal $L$. This result is obtained by using the quotient space $\mathrm{pSpec}\ R$ of the prime spectrum of the ring $R$ modulo the equivalence generated by the inclusion. When every prime ideal contains only one minimal prime, for instance if $R$ is arithmetical, $\mathrm{pSpec}\ R$ is Hausdorff and there is a bijection between this quotient space and the minimal prime spectrum $\mathrm{Min}\ R$, which is a homeomorphism if and only if $\mathrm{Min}\ R$ is compact. If $x$ is a closed point of $\mathrm{pSpec}\ R$, there is a pure ideal $A(x)$ such that $x=V(A(x))$. If $R$ is almost clean, i.e. each element is the sum of a regular element with an idempotent, it is  shown that $\mathrm{pSpec}\ R$ is totally disconnected and, $\forall x\in\mathrm{pSpec}\ R$, $R/A(x)$ is almost clean; the converse holds if every principal ideal is finitely presented. Some questions posed by Facchini and Faith at the second International Fez Conference on Commutative Ring Theory in 1995, are also investigated. If $R$ is a commutative ring for which the ring $Q(R/A)$ of quotients of $R/A$ is an IF-ring for each proper ideal $A$, it is proved that $R_P$ is a strongly discrete valuation ring for each maximal ideal $P$ and $R/A$ is semicoherent for each proper ideal $A$.
\end{abstract}

\maketitle

%********************************************%
% Please do not alter the lines in this box. %
%********************************************%
%\firstpage{1}                                %
%********************************************%

%*******************************%
% Please input your paper here. %
%*******************************%
\section{Introduction}
In this paper we consider the following two  questions:
\begin{question}
\label{Q:edr1}
 Is every B\'{e}zout domain an elementary divisor ring?
\end{question}

\begin{question}
\label{Q:edr2}
More generally, is every B\'{e}zout semihere\-ditary ring an
  elementary divisor ring?
\end{question}
The first question was posed by M. Henriksen in \cite{Hen55} in 1955, and
the second by M.D. Larsen, W.J Lewis and T.S. Shores in \cite{LLS74} in
1974.

In Section~\ref{S:arithmetic} we prove that these two questions are equivalent but they are still unsolved.

To show this equivalence, we use the quotient space $\mathrm{pSpec}\ R$ of $\mathrm{Spec}\ R$ modulo the equivalence generated by the inclusion, where $R$ is a commutative ring. When $R$ is a {\it Gelfand ring}, i.e. each prime ideal is contained in only one maximal,  $\mathrm{pSpec}\ R$ is Hausdorff and homeomorphic to $\mathrm{Max}\ R$ (Proposition~\ref{P:Gel}). On the  other hand, if each prime ideal contains a unique minimal prime, then $\mathrm{pSpec}\ R$ is Hausdorff and there is a continuous bijection from $\mathrm{Min}\ R$ into $\mathrm{pSpec}\ R$  which is a homeomorphism if and only if  $\mathrm{Min}\ R$ is compact. There is also a continuous surjection $\tau_R:\mathrm{pSpec}\ R\rightarrow \mathrm{Spec}\ \mathrm{B}(R)$, where $\mathrm{B}(R)$ is the Boolean ring associated to $R$, and $\tau_R$ is a homeomorphism if and only if $\mathrm{pSpec}\ R$ is totally disconnected. In this case, it is possible to get some interesting algebraic results by using Lemma~\ref{L:poly}.

A  ring $R$ is said to be \textit{clean} (respectively \textit{almost clean} (see \cite{McG03}))  if each element of $R$ is the sum of an idempotent with a unit (respectively a regular element). In Section~\ref{S:clean} we show that the total disconnectedness of $\mathrm{pSpec}\ R$ is necessary if the ring $R$ is almost clean. Recall that a ring $R$ is clean if and only if $R$ is Gelfand and $\mathrm{Max}\ R$ totally disconnected: see \cite[Theorem 1.7]{McG06}, \cite[Corollary 2.7]{LuYu06} or \cite[Theorem I.1]{Cou07}.
Almost clean rings were introduced by McGovern in 
\cite{McG03} and studied by several authors: Ahn and Anderson \cite{AhAn06}, Burgess and Raphael \cite{BuRa08} and \cite{BuRap08},  Varadarajan \cite{Var07}. If $Q$ is the quotient ring of $R$ and if each prime ideal of $R$ contains a unique minimal prime, we show that $\mathrm{pSpec}\ R$ and $\mathrm{pSpec}\ Q$ are homeomorphic and $\mathrm{B}(R)=\mathrm{B}(Q)$; moreover, if $R$ is arithmetical, then $R$ is almost clean if $Q$ is clean, and the converse holds if $Q$ is coherent.

In Section~\ref{S:IF} we  give partial answers to some questions posed by Facchini and Faith at the second International Fez Conference on Commutative Ring Theory in 1995 \cite{FaFa97}. If $R$ is fractionally IF, it is shown that $R/A$ is semicoherent for each ideal $A$ and $R_P$ is a strongly discrete valuation ring for each maximal ideal $P$. We give an example of a finitely fractionally self FP-injective ring which is not arithmetical; recall that Facchini and Faith proved that each fractionally self FP-injective ring is arithmetical. It is also proven that any ring which is either clean, coherent and arithmetical or semihereditary is finitely fractionally IF. However, there exist examples of clean coherent arithmetical rings with a non-compact minimal prime spectrum; recall that the author proved that $\mathrm{Min}\ R/A$ is compact for any ideal $A$ of a fractionally self FP-injective ring $R$.

All rings in this paper are associative and commutative with unity, and all modules are
unital.We denote respectively $\mathrm{Spec}\ R$, $\mathrm{Max}\ R$ and $\mathrm{Min}\ R,$ the
space of prime ideals, maximal ideals and minimal prime ideals of
$R$, with the Zariski topology. If $A$ a subset of $R$, then we denote  
$V(A) = \{ P\in\mathrm{Spec}\ R\mid A\subseteq P\}\ \mathrm{and}\ \ D(A) =\mathrm{Spec}\ R\setminus V(A).$

\section{A quotient space of the prime spectrum of a ring}
\label{S:pSpec}
If $R$ is a ring, we consider on $\mathrm{Spec}\ R$ the equivalence relation $\mathcal{R}$ defined by   $L\mathcal{R} L'$ if there exists a finite sequence of prime ideals $(L_k)_{1\leq k\leq n}$ such that $L=L_1,$ $L'=L_n$ and $\forall k,\ 1\leq k\leq (n-1),$ either $L_k\subseteq L_{k+1}$ or $L_k\supseteq L_{k+1}$. We denote by $\mathrm{pSpec}\ R$ the quotient space of $\mathrm{Spec}\ R$ modulo $\mathcal{R}$ and by $\lambda_R: \mathrm{Spec}\ R\rightarrow\mathrm{pSpec}\ R$ the natural map. The quasi-compactness of $\mathrm{Spec}\ R$ implies the one of $\mathrm{pSpec}\ R$, but generally $\mathrm{pSpec}\ R$  is not  $T_1$: see \cite[Propositions 6.2 and 6.3]{Laz67}. However:

\begin{proposition}
\label{P:Gel} The following conditions are equivalent for a ring $R$:
\begin{enumerate}
\item the restriction of $\lambda_R$ to $\mathrm{Max}\ R$ is a homeomorphism;
\item the restriction of $\lambda_R$ to $\mathrm{Max}\ R$ is injective;
\item $R$ is Gelfand.
\end{enumerate} 
In this case $\mathrm{pSpec}\ R$ is Hausdorff.
\end{proposition}
\begin{proof}
It is obvious that $(i)\Rightarrow (ii)$.

$(ii)\Rightarrow (iii)$. If a prime ideal is contained in two maximals ideals $P_1$ and $P_2$ we get that $\lambda_R(P_1)=\lambda_R(P_2)$.

$(iii)\Rightarrow (i)$. If $L$ is a prime ideal we denote by $\mu(L)$ the unique maximal ideal containing $L$. It is easy to verify that $\mu(L)=\mu(L')$ if $L\mathcal{R}L'$. So, $\mu$ induces a map $\bar{\mu}:\mathrm{pSpec}\ R\rightarrow\mathrm{Max}\ R$. We easily show that $\bar{\mu}^{-1}=\lambda_R\vert_{\mathrm{Max}\ R}$. By \cite[Theorem 1.2]{MaOr71} $\mu$ is continuous and $\mathrm{Max}\ R$ is Hausdorff. Hence $\bar{\mu}$ is a homeomorphism and $\mathrm{pSpec}\ R$ is Hausdorff.
\end{proof}

\begin{proposition}[{\cite[Proposition IV.1]{Cou07}}]
\label{P:Hausd} Let $R$ be a ring such that each prime ideal contains only one minimal prime. Then $\mathrm{pSpec}\ R$  is  Hausdorff and $\lambda_R\vert_{\mathrm{Min}\ R}$ is bijective. Moreover $\lambda_R\vert_{\mathrm{Min}\ R}$ is a homeomorphism if and only if $\mathrm{Min}\ R$ is compact.
\end{proposition}

\begin{proposition}
\label{P:contmorp} Let $\varphi:R\rightarrow T$ be a ring homomorphism. Then $\varphi$ induces a continuous map $^b\varphi:\mathrm{pSpec}\ T\rightarrow\mathrm{pSpec}\ R$ such that $\lambda_R\circ {}^a\varphi={}^b\varphi\circ\lambda_T$, where $^a\varphi:\mathrm{Spec}\ T\rightarrow\mathrm{Spec}\ R$ is the continuous map induced by $\varphi$.
\end{proposition}
\begin{proof} If $L$ and $L'$ are prime ideals of $T$ such that $L\subseteq L'$ then $^a\varphi(L)\subseteq {}^a\varphi(L')$. Hence, if $x\in\mathrm{pSpec}\ T$, we can put $^b\varphi(x)=\lambda_R({}^a\varphi(L))$ where $L\in x$. Since $\lambda_R$ and ${}^a\varphi$ are continuous, so is $^b\varphi$.  \end{proof}

\medskip
An exact sequence $0 \rightarrow F \rightarrow E \rightarrow G \rightarrow 0$  is {\it pure}
if it remains exact when tensoring it with any $R$-module. Then, we say that $F$ is a \textit{pure} submodule of $E$. By \cite[Proposition 8.6]{FuSa01} $F$ is a pure submodule of $E$ if every finite system of equations
\[\sum_{i=1}^{i=n}r_{j,i}x_i=y_j\in F,\qquad (1\leq j\leq p),\]
with coefficients $r_{j,i}\in R$ and unknowns $x_1,\dots,x_n$, has a solution in $F$ whenever it is solvable in $E$. The following proposition is well known.
\begin{proposition}
\label{P:purIdeal} Let $A$ be an ideal of a ring $R$. The following conditions are equivalent:
\begin{enumerate}
\item $A$ is a pure ideal of $R$;
\item for each finite family $(a_i)_{1\leq i\leq n}$ of elements of $A$ there exists $t\in A$ such that $a_i=a_it,\ \forall i,\ 1\leq i\leq n$;
\item for all $a\in A$ there exists $b\in A$ such that $a=ab$;
\item $R/A$ is a flat $R$-module.
\end{enumerate} 
Moreover, if $A$ is finitely generated, then $A$ is pure if and only if it is generated by an idempotent.
\end{proposition}
\begin{proof}
$(ii)\Rightarrow (iii)$ is obvious.

$(iii)\Rightarrow (iv)$. Let $B$ be an ideal of $R$. We must prove that $A\cap B=AB$. If $a\in A\cap B$, there exists $t\in A$ such that $a=at$. Hence $a\in AB$.

$(iv)\Rightarrow (iii)$. If $a\in A$, then $Ra=A\cap Ra=Aa$ by $(iv)$.

$(i)\Rightarrow (iii)$. If $a\in A$, $1$ is solution of the equation $ax=a$. So, this equation has a solution in $A$.

$(iii)\Rightarrow (ii)$. Let $a_1,\dots,a_n$ be elements of $A$. We proceed by induction on $n$. There exists $t\in A$ such that $a_n=ta_n$. By induction hypothesis there exists $s\in A$ such that $a_i-ta_i=s(a_i-ta_i),\ \forall i,\ 1\leq i\leq (n-1)$. Now, it is easy to check that $(s+t-st)a_i=a_i,\ \forall i,\ 1\leq i\leq n$.

$(ii)\Rightarrow (i)$. We consider the following system of equations:
\[\sum_{i=1}^{i=n}r_{j,i}x_i=a_j\in A,\ 1\leq j\leq p.\]
Assume that $(c_1,\dots,c_n)$ is a solution of this system in $R$. There exists $s\in A$ such that $a_j=sa_j,\ \forall j,\ 1\leq j\leq p$. So, $(sc_1,\dots,sc_n)$ is a solution of this system in $A$.
\end{proof}

We set  $0_P$  the kernel of the natural map $R\rightarrow R_P$ where $P\in\mathrm{Spec}\ R$.

\begin{lemma}
\label{L:pure} Let $R$ be a ring and let $C$ a closed subset of $\mathrm{Spec}\ R$. Then $C$ is the inverse image of a closed subset of $\mathrm{pSpec}\ R$ by $\lambda_R$ if and only if $C=V(A)$ where $A$ is a pure ideal. Moreover, in this case, $A=\cap_{P\in C}0_P$.
\end{lemma}
\begin{proof} Let $A$ be a pure ideal, and let $P$ and $L$ be prime ideals such that $A\subseteq P$ and $L\subseteq P$. Since $A$ is pure, for each $a\in A$ there exists $b\in A$ such that $a=ab$. Then $(1-b)a=0$ and $(1-b)\notin P$, whence  $(1-b)\notin L$ and $a\in L$. So, $L\in V(A)$ and $V(A)$ is the inverse image of a closed subset of $\mathrm{pSpec}\ R$ by $\lambda_R$.

Let $C=V(B)$ where $B=\cap_{L\in C}L$. Suppose that $C$ is the inverse image of a closed subset of $\mathrm{pSpec}\ R$ by $\lambda_R$. We put $A=\cap_{P\in C}0_P$. Let $b\in B$ and $P\in C$. Then $C$ contains each minimal prime ideal contained in $P$. So, the image of $b$, by the natural map $R\rightarrow R_P$, belongs to the nilradical of $R_P$. It follows that there exist  $0\ne n_P\in\mathbb{N}$ and $s_P\in R\setminus P$ such that $s_Pb^{n_P}=0$. Hence, $\forall L\in D(s_P)\cap C,\ b^{n_P}\in 0_L$. A finite family $(D(s_{P_j}))_{1\leq j\leq m}$ covers $C$. Let $n=\max\{n_{P_1},\dots,n_{P_m}\}$. Then $b^n\in 0_L,\ \forall L\in C$, whence $b^n\in A$. We deduce that $C=V(A)$. Now, we have $A_P=0$ if $P\in V(A)$ and $A_P=R_P$ if $P\in D(A)$. Hence $A$ is a pure ideal.  \end{proof} 

\begin{corollary}
\label{C:connected} For any ring $R$ the following assertions hold:
\begin{enumerate}
\item a subset $U$ of $\mathrm{pSpec}\ R$ is open and closed if and only if there exists an idempotent $e\in R$ such that $\lambda_R^{\leftarrow}(U)=D(e)$;
\item $R$ is indecomposable if and only if $\mathrm{pSpec}\ R$ is connected.
\end{enumerate} 
\end{corollary}
\begin{proof}
A subset $U$ of $\mathrm{pSpec}\ R$ is open and closed if and only if is so $\lambda_R^{\leftarrow}(U)$ and it is well known that a subset $U'$ of $\mathrm{Spec}\ R$ is open and closed if and only if $U'=D(e)$ for some idempotent $e\in R$. The second assertion is an immediate consequence of the first.
\end{proof}

\bigskip

If $\{x\}$ is closed in $\mathrm{pSpec}\ R$ we denote by $A(x)$ the pure ideal of $R$ for which $x=\mathrm{V}(A(x))$. A topological space is called \textit{totally disconnected} if each of its connected components contains only one point. Every Hausdorff topological space $X$ with a base of clopen neighbourhoods is totally disconnected and the converse holds if $X$ is compact (see \cite[Theorem 16.17]{GiJe60}).

\begin{proposition} \label{P:discon}
Let $R$ be a ring. Then  the following conditions are equivalent:
\begin{enumerate}
\item $\mathrm{pSpec}\ R$ is  totally disconnected;
\item for each $x\in \mathrm{pSpec}\ R$, $\{x\}$ is closed and $A(x)$ is generated by idempotents.
\end{enumerate} 
\end{proposition}
\begin{proof} 
$(ii)\Rightarrow (i)$. Let $x,y\in \mathrm{pSpec}\ R$, $x\ne y$. Then $V(A(x))\cap V(A(y))=\emptyset$. So, $A(x)+A(y)=R$, whence $\exists a\in A(x)$ such that $(1-a)\in A(y)$. There exists an idempotent $e\in A(x)$ such that $a=ae$. So, $(1-e)(1-a)=(1-e)\in A(y)$. We easily deduce  that $x\subseteq D(1-e)$ and $y\subseteq D(e)$.

$(i)\Rightarrow (ii)$. Let $x\in \mathrm{pSpec}\ R$ and $a\in A(x)$. There exists $b\in A(x)$ such that $a=ab$. So $(1-b)a=0$. Clearly $x\subseteq D(1-b)$. Since $\mathrm{pSpec}\ R$ is Hausdorff and $\mathrm{Spec}\ R$ is quasi-compact, $\lambda_R^{\rightarrow}(V(1-b))$ is closed. Therefore $U=\mathrm{pSpec}\ R\setminus \lambda_R^{\rightarrow}(V(1-b))$ is open and contains $x$. The condition $\mathrm{pSpec}\ R$ is totally disconnected implies that there exists an idempotent $e$ such that $x\subseteq D(e)\subseteq\lambda_R^{\leftarrow}(U)\subseteq D(1-b)$. 	If follows that $e\in R(1-b)$. So $ea=0$ and consequently $a=a(1-e)$. From $x\subseteq D(e)$ and $e(1-e)=0$ we deduce that $(1-e)\in A(x)$ by Lemma~\ref{L:pure}.
\end{proof}

\medskip For any ring $R$, $\mathrm{B}(R)$ is the set of idempotents of $R$. For any $e,e'\in\mathrm{B}(R)$ we put $e\oplus e'=e+e'-ee'$ and $e\odot e'=ee'$. With these operations $\mathrm{B}(R)$ is a Boolean ring. The space $\mathrm{Spec}\ \mathrm{B}(R)$ is denoted by $\mathrm{X}(R)$. Then $\mathrm{X}(R)$ is Hausdorff, compact and totally disconnected. If $x\in\mathrm{X}(R)$ the {\it stalk} of $R$ at $x$ is the quotient of $R$ by the ideal generated by the idempotents contained in $x$.
 
\begin{proposition} \label{P:stalk} Let $R$ be a ring. The following assertions hold:
\begin{enumerate}
\item  there exists a surjective continuous map $\tau_R:\mathrm{pSpec}\ R\rightarrow\mathrm{X}(R)$;
\item $\mathrm{pSpec}\ R$ is totally disconnected if and only if $\tau_R$ is a homeomorphism. In this case, for each $x\in\mathrm{pSpec}\ R$, $R/A(x)$ is the stalk of $R$ at $\tau_R(x)$.
\end{enumerate} 
\end{proposition}
\begin{proof} $(i)$. If $L$ and $L'$ are prime ideals of $R$, $L\subseteq L'$, then $L\cap\mathrm{B}(R)=L'\cap\mathrm{B}(R)$ since each prime ideal of $\mathrm{B}(R)$ is maximal. So, $\tau_R$ is well defined. It is easy to check that for any idempotent $e\in R$, $\tau_R^{\leftarrow}(D(e))=\lambda_R^{\rightarrow}(D(e))$. Hence $\tau_R$ is continuous. For each $x\in\mathrm{X}(R)$, if $L$ is a maximal ideal containing all elements of $x$, then $x=\tau_R(\lambda_R(L))$, whence $\tau_R$ is surjective.

$(ii)$. It is obvious that $\mathrm{pSpec}\ R$ is totally disconnected if $\tau_R$ is a homeomorphism. Conversely, since  $\mathrm{pSpec}\ R$ is compact and $\mathrm{X}(R)$ is Hausdorff it is enough to show that $\tau_R$ is injective. Let $x,x'\in\mathrm{pSpec}\ R,\ x\ne x'$. There exists an idempotent $e$ such that $x\in\lambda_R^{\rightarrow}(D(e))$ and  $x'\in\lambda_R^{\rightarrow}(D(1-e))$. It follows that $e\notin\tau_R(x)$ and $e\in\tau_R(x')$. Hence $\tau_R(x)\ne\tau_R(x')$. The last assertion is a consequence of Proposition~\ref{P:discon} and Lemma~\ref{L:pure}.
\end{proof}

\medskip

The following lemma will be useful to show some important results of this paper.

\begin{lemma} \label{L:poly}
Let $R$ be a ring such that $\mathrm{pSpec}\ R$ is totally disconnected. Then any $R$-algebra 
  $S$ (which is not necessarily commutative) satisfies the following condition: 
let
  $f_1,\dots,f_k$ be polynomials over $S$ in noncommuting variables
  $x_1,\dots,x_m,y_1,\dots,y_n$. Let $a_1,\dots,a_m\in S$. Assume that,
   $\forall x\in \mathrm{pSpec}\ R$ there exist $b_1,\dots,b_n\in S$
  such that:
   
   \centerline{$f_i(a_1,\dots,a_m,b_1,\dots,b_n)\in A(x)S$,  $\forall i$, $1\leq
  i\leq k$.}
   Then there exist $d_1,\dots,d_n\in S$ such that:
  
  \centerline{$f_i(a_1,\dots,a_m,d_1,\dots,d_n)=0$,  $\forall i$, $1\leq
  i\leq k$.} 
\end{lemma}
\begin{proof} Let $x\in \mathrm{pSpec}\ R$ and let $b_{x,1},\dots,b_{x,n}\in S$
  such that 
  \[f_i(a_1,\dots,a_m,b_{x,1},\dots,b_{x,n})\in A(x)S,\  \forall i,\ 1\leq i\leq k.\]
  Then there exists a finitely generated ideal $A\subseteq A(x)$ such that \[f_i(a_1,\dots,a_m,b_{x,1},\dots,b_{x,n})\in AS,\  \forall i,\ 1\leq i\leq k.\]
   There exists an idempotent $e_x$ such that $A\subseteq R(1-e_x)\subseteq A(x)$. Hence \[e_xf_i(a_1,\dots,a_m,b_{x,1},\dots,b_{x,n})=0,\  \forall i,\ 1\leq i\leq k.\] 
A finite family $(\lambda_R^{\rightarrow}(D(e_{x_j})))_{1\leq j\leq p}$ covers $\mathrm{pSpec}\ R$. We may assume that $(e_{x_j})_{1\leq j\leq p}$ is a family of orthogonal idempotents. We put $d_{\ell}=e_{x_1}b_{x_1,\ell}+\dots+e_{x_p}b_{x_p,\ell},\ \forall\ell,\ 1\leq\ell\leq n.$ Then $f_i(a_1,\dots,a_m,d_1,\dots,d_n)=0$,  $\forall i$, $1\leq
  i\leq k$.  \end{proof}
  
We denote by $\mathrm{gen}\ M$  the minimal number of generators of a finitely generated $R$-module $M$. The following proposition is an example of an algebraic result that can be proven  by using  Lemma~\ref{L:poly}. Recall that the {\it trivial extension} $R\ltimes M$ of $R$ by $M$ is defined by:
\(R\ltimes M=\{\binom{r\ \ x}{0\ \ r}\mid r\in R,\ x\in M\}.\)
It is convenient to identify $R\ltimes M$ with the $R$-module $R\oplus M$ endowed with the following multiplication: \((r,x)(s,y)=(rs,ry+sx)\), where $r,s\in R$ and $x,y\in M$.
\begin{proposition} \label{P:gen} Let $R$ be a ring such that $\mathrm{pSpec}\ R$ is totally disconnected. Let $M$ be a finitely generated $R$-module and $F$ a finitely presented $R$-module. Then:
\begin{enumerate}
\item if, $\forall x\in\mathrm{pSpec}\ R$, $M/A(x)M$ is a homomorphic image of $F/A(x)F$, then $M$ is a homomorphic image of $F$;
\item $\mathrm{gen}\ M=\sup\{\mathrm{gen}\ (M/A(x)M)\mid x\in\mathrm{pSpec}\ R\};$
\item if $M$ is finitely presented and, $\forall x\in\mathrm{pSpec}\ R$, $M/A(x)M\cong F/A(x)F$, then $M\cong F$.
\end{enumerate} 
\end{proposition}
\begin{proof}
$(i)$. Let  $\{m_1,\dots,m_p\}$ be a spanning set of $M$. Let $\{f_1,\dots,f_n\}$ be a spanning set of $F$ with the following relations: $\forall\ell, 1\leq\ell\leq n'$, $\sum_{i=1}^{i=n}c_{l,i}f_i=0$. We put $S=R\ltimes M$ the trivial extension of $R$ by $M$. We consider the following system $\mathcal{E}_1$ of polynomial equations in variables $X_{j,i},\ Y_i, Z_{i,j},\ 1\leq j\leq k,\ 1\leq i\leq n$:
\[\sum_{i=1}^{i=n}X_{j,i}Y_i=(0,m_j),\ \forall j,\ 1\leq j\leq k;\quad Y_i=\sum_{j=1}^{j=p}Z_{i,j}(0,m_j),\ \forall i,\ 1\leq i\leq n;\]
\[\sum_{i=1}^{i=n}(c_{l,i},0)Y_i=0,\ \forall\ell,\ 1\leq\ell\leq n'.\]
Thus $\mathcal{E}_1$ has a solution modulo $A(x)S$ for each $x\in\mathrm{pSpec}\ R$. By Lemma~\ref{L:poly} $\mathcal{E}_1$ has a solution $x_{j,i},\ y_i, z_{i,j},\ 1\leq j\leq p,\ 1\leq i\leq n$ in $S$. It is easy to check that $y_i=(0,m'_i),\ \forall i,\ 1\leq i\leq n$, and if $x_{j,i}=(r_{j,i},x_{j,i}'),\forall j,i, \ 1\leq j\leq p,\ 1\leq i\leq n$, then $m_j=\sum_{i=1}^{i=n}r_{j,i}m_i',\ \forall j,1\leq j\leq p$. We have also: $\forall\ell, 1\leq\ell\leq n'$, $\sum_{i=1}^{i=n}c_{l,i}m_i'=0$. Hence we get an epimorphism $\phi:F\rightarrow M$ defined by $\phi(f_i)=m'_i,\ \forall i,\ 1\leq i\leq n$.

$(ii)$ is an easy consequence of $(i)$.

$(iii)$. Let the notations be as in $(i)$. We assume that $m_1,\dots,m_p$ verify the follo\-wing relations: $\forall k,\ 1\leq k\leq p',\ \sum_{j=1}^{j=p}d_{k,j}m_j=0$. 

Observe that $F\cong M$ if $M$ has a spanning set $\{m'_1,\dots,m'_n\}$ with the relations: $\forall\ell,\ 1\leq\ell\leq n',\ \sum_{i=1}^{i=n}c_{l,i}m'i=0$. In this case there exist $r_{j,i}\in R$ such that $m_j=\sum_{i=1}^{i=n}r_{j,i}m'_i,\ \forall j,\ 1\leq j\leq p$. Thus \[\sum_{j=1}^{j=p}d_{k,j}m_j=\sum_{i=1}^{i=n}\left( \sum_{j=1}^{j=p}d_{k,j}r_{j,i}\right) m'_i=0,\ \forall k,\ 1\leq k\leq p'.\]
It follows that there exist $w_{k,l}\in R$ such that:
\[\sum_{j=1}^{j=p}d_{k,j}m_j=\sum_{\ell=1}^{\ell=n'}w_{k,l}\left( \sum_{i=1}^{i=n}c_{\ell,i}m'_i\right) =\sum_{i=1}^{i=n}\left( \sum_{\ell=1}^{\ell=n'}w_{k,l}c_{l,i}\right) m'i=0,\ \forall k,\ 1\leq k\leq p'.\]
We deduce that 
\begin{equation}\label{E:rel}
\sum_{j=1}^{j=p}d_{k,j}r_{j,i}=\sum_{\ell=1}^{\ell=n'}w_{k,l}c_{l,i},\ \forall k,\ 1\leq k\leq p',\ \forall i,\ 1\leq i\leq n.
\end{equation}
Conversely, if there exists an epimorphism $\phi:F\rightarrow M$ defined by $\phi(f_i)=m'_i$, then $\phi$ is bijective if each relation $(rel)$ $\sum_{i=1}^{i=n}a_im'_i=0$ is a linear combination of the relations $\sum_{i=1}^{i=n}c_{\ell,i}m'_i=0$. Since $m'_i$ is a linear combination of $m_1,\dots,m_p$, from the relation $(rel)$ we get a relation $(rel1)$ which is a linear combination of the relations $\sum_{j=i}^{j=p}d_{k,j}m_j=0$. By using  the equalities~\ref{E:rel}, we get that $(rel)$ is a linear combination of the relations $\sum_{i=1}^{i=n}c_{\ell,i}m'_i=0$.

Let $\mathcal{E}_2$  be the system  of polynomial equations in variables $X_{j,i},\ W_{k,\ell}$, $1\leq j\leq p$, $1\leq i\leq n$, $1\leq k\leq p'$, $1\leq\ell\leq n'$, \[\sum_{j=1}^{j=p}(d_{k,j},0)X_{j,i}=\sum_{\ell=1}^{\ell=n'}W_{k,l}(c_{\ell,i},0),\ \forall k,\ 1\leq k\leq p',\ \forall i,\ 1\leq i\leq n.\]
We put $\mathcal{E}=\mathcal{E}_1\cup\mathcal{E}_2$. As in $(i)$ and by using the above observation we show that $\mathcal{E}$ has a solution. We define $\phi:F\rightarrow M$ as in $(i)$, and by using the fact that $\mathcal{E}_2$ has a solution, we prove that $\phi$ is injective by using the above observation.
\end{proof}

\section{Hermite rings and elementary divisor rings} 
\label{S:arithmetic}
An $R$-module is called {\it uniserial} if if the set of its submodules is totally ordered by inclusion. A ring $R$ is a {\it valuation ring} if it is a uniserial $R$-module.
We say that $R$ is \textit{arithmetical} if $R_L$ is a valuation ring for each maximal ideal $L$.
A ring
is a {\it B\'ezout ring} if every finitely generated ideal is
principal. A ring $R$ is {\it Hermite} if $R$ 
satisfies the following property~: for every $(a,b) \in R^2$, there
exist $d, a', b'$ in $R$ such that $a = da'$, $b = db'$ and
$Ra' + Rb' = R$. We say that $R$ is an {\it elementary divisor ring} if for every matrix $A$,
with entries in $R$, there  exist
a diagonal matrix $D$ and invertible matrices $P$ and $Q$, with
entries in $R$, such that $PAQ = D$.  Then we have the following
implications:

\centerline{elementary divisor ring $\Rightarrow$ Hermite ring
$\Rightarrow$ B\'ezout ring $\Rightarrow$ arithmetical ring;}
but these implications are not reversible: see \cite{GiHen56} or \cite{Car87}.

\begin{theorem}
\label{T:Hermite} Let $R$ be a ring such that $\mathrm{pSpec}\ R$ is  totally disconnected. Assume that $R/A(x)$ is B\'ezout for each $x\in\mathrm{pSpec}\ R$. Then $R$ is Hermite.
\end{theorem}
\begin{proof} By Proposition~\ref{P:gen} $R$ is B\'ezout. It follows that each prime ideal contains a unique minimal prime, whence $R/A(x)$ has a unique minimal prime ideal. By \cite[Theorem 2]{Hen55}, a ring with a unique minimal ideal is Hermite if and only if it is B\'ezout. So, $R/A(x)$ is Hermite. Now, let $a,\ b\in R$. Consider the following polynomial equations: $a=XZ,\ b=YZ$ and $1=SX+TY$. For each $x\in\mathrm{pSpec}\ R$, these equations have a solution modulo $A(x)$. By Lemma~\ref{L:poly} they have a solution in $R$.  \end{proof}

\bigskip
Recall that a ring $R$ is  \textit{(semi)hereditary} if each (finitely generated) ideal is projective. If $F$ is a submodule of a module $E$ and $x$ an element of $E$, then the ideal $\{r\in R\mid rx\in F\}$ is denoted by $(F:x)$.

The following was already proved, see \cite[Theorem III.3]{Cou03} and \cite[Theorem 2.4]{LLS74}. 

\begin{corollary}
\label{C:semiher} Let $R$ be a B\'ezout ring.
Then the following assertions hold:
\begin{enumerate}
\item  $R$ is Hermite if $\mathrm{Min}\ R$ is compact;
\item $R$ is Hermite if it is semihereditary.
\end{enumerate}  
\end{corollary}
\begin{proof} $(i)$.  Since each prime ideal contains a unique minimal prime, $\lambda_R\vert_{\mathrm{Min}\ R}$ is bijective. Moreover, by \cite[Proposition IV.1]{Cou07} $\mathrm{pSpec}\ R$ is Hausdorff.  It follows that $\lambda_R\vert_{\mathrm{Min}\ R}$ is a homeomorphism because $\mathrm{Min}\ R$ is compact. We can apply the previous theorem because  $\mathrm{Min}\ R$ is always totally disconnected (\cite[Corollary 2.4]{HeJe65}):\\ \centerline{$\forall a\in R,$ $D(a)\cap\mathrm{Min}\ R=V((N:a))\cap\mathrm{Min}\ R$ where $N$ is the nilradical of $R$.}

$(ii)$ is an immediate consequence of $(i)$ because $\mathrm{Min}\ R$ is compact if $R$ is semihereditary by \cite[Proposition 10]{Que71}.  \end{proof}

\medskip
The following example shows that $\mathrm{pSpec}\ R$ is not generally totally disconnected, even if $R$ is arithmetical.

\begin{example}
Consider \cite[Example 6.2 (due to Jensen)]{Vas76} defined in the following way: let $\mathcal{I}$ be a family of pairwise disjoint intervals of the real line with rational endpoints, such that between any two intervals of $\mathcal{I}$ there is at least another interval of $\mathcal{I}$; let $R$ be the ring of continuous maps $\mathbb{R}\to\mathbb{R}$ which are rational constant by interval except on finitely many intervals of $\mathcal{I}$ on which it is given by a rational polynomial. It is easy to check that $R$ is a reduced indecomposable  ring. It is also B\'ezout (left as an exercise!). Let $[a,b]\in\mathcal{I}$ et $f\in R$ defined by $f(x)=(x-a)(b-x)$ if $a\leq x\leq b$ and $f(x)=0$ elsewhere. Then $(0:f)$ is not finitely generated, whence $R$ is not semihereditary. So, $\mathrm{pSpec}\ R$ is an infinite set and a compact connected topological space.
\end{example}
\begin{theorem}
\label{T:EDR} Let $R$ be a ring such that $\mathrm{pSpec}\ R$ is totally disconnected. Assume that $R/A(x)$ is B\'ezout for each $x\in\mathrm{pSpec}\ R$. Then $R$ is an elementary divisor ring if and only if so is $R/L$, for each minimal prime ideal $L$.
\end{theorem}
\begin{proof} Only "if" requires a proof. By Theorem~\ref{T:Hermite} $R$ is Hermite. Let $x\in\mathrm{pSpec}\ R$ and $R'=R/A(x)$. Then $R'$ has a unique minimal prime ideal. Let $L$ be the minimal prime ideal of $R$ such that $L/A(x)$ is the minimal prime of $R'$. Thus $L/A(x)$ is contained in the Jacobson radical $\mathcal{J}(R')$ of $R'$. So, $R'/\mathcal{J}(R')$ is an elementary divisor ring since it is a homomorphic image of $R/L$. By \cite[Theorem 3]{Hen55} a Hermite ring $S$ is an elementary divisor ring if and only if so is $S/\mathcal{J}(S)$. Hence $R/A(x)$ is an elementary divisor ring. Let $a, b, c\in R$ such that $Ra+Rb+Rc=R$. We consider the polynomial equation in variables $X, Y, S, T$: $aSX+bTX+cTY=1$. By \cite[Theorem 6]{GiHe56}, this equation has a solution modulo $A(x)$,  $\forall x\in\mathrm{pSpec}\ R$. So, by Lemma~\ref{L:poly} there is a solution in $R$. We conclude by \cite[Theorem 6]{GiHe56}.  \end{proof}

\bigskip
With a similar proof as in Corollary~\ref{C:semiher}, we get Corollary~\ref{C:equiv}. The second condition shows that the two questions~\ref{Q:edr1} and \ref{Q:edr2} have the same answer.

\begin{corollary}
\label{C:equiv} 
The following assertions hold:
\begin{enumerate}
\item Let $R$ be a B\'ezout ring with compact minimal prime spectrum. Then $R$ is an elementary divisor ring if and only if so is $R/L$, for each minimal prime ideal $L$.
\item Let $R$ be a semihereditary ring. Then $R$ is an elementary divisor ring if and only if so is $R/L$, for each minimal prime ideal $L$.
\item Let $R$ be a hereditary ring. Then $R$ is an elementary divisor ring if and only if  $R/L$ is B\'ezout for each minimal prime ideal $L$.
\end{enumerate} 
\end{corollary}
\smallskip

The third assertion can be also deduced from \cite[Corollary]{Sho74}.

\section{Almost clean rings}
\label{S:clean}
In \cite[Proposition 15]{McG03}, McGovern proved that each element of a ring $R$ is the product of an idempotent with a regular element if and only if $R$ is a \textit{PP-ring}, i.e. each principal ideal is projective, and he showed that each PP-ring is almost clean (\cite[Proposition 16]{McG03}). The aim of this section is to study almost clean rings.

In the sequel, if $R$ is a ring, $\mathfrak{R}(R)$ is its set of regular elements of $R$ and \[\Phi_R=\{L\in\mathrm{Spec}\ R\mid L\cap\mathfrak{R}(R)=\emptyset\}.\] By \cite[Corollaire 2 p.92]{Bou61} each zero-divisor is contained in an element of $\Phi_R$. So, the following proposition is obvious.

\begin{proposition} \label{P:loalclean} Let $R$ be a ring. The following conditions are equivalent:
\begin{enumerate}
\item For each $a\in R$, either $a$ or $(a-1)$ is regular.
\item $R$ is almost clean and indecomposable.
\item $\forall L,L'\in\Phi_R,\ L+L'\ne R$.
\end{enumerate}
\end{proposition}

\begin{corollary} \label{C:fracValu}
Let $R$ be an arithmetical ring, $Q$ its ring of fractions and $N$ its nilradical. Then:
\begin{enumerate}
\item $R$ is almost clean and indecomposable if and only if $Q$ is a valuation ring;
\item $R/A$ is almost clean and indecomposable for each ideal $A\subseteq N$ if and only if $N$ is prime and uniserial.
\end{enumerate} 
\end{corollary}
\begin{proof} $(i)$. Assume that $R$ is almost clean and indecomposable. Then if $L,L'\in\Phi_R$ then there exists a maximal ideal $P$ such that $L+L'\subseteq P$. Since $R_P$ is a valuation ring, either $L\subseteq L'$ or $L'\subseteq L$. By \cite[Corollaire p.129]{Bou61} $\Phi_R$ is homeomorphic to $\mathrm{Spec}\ Q$. It follows that $Q$ is local. Conversely, $\Phi_R$ contains a unique maximal element. 
 
$(ii)$. First, assume that $R/A$ is almost clean and indecomposable for each ideal $A\subseteq N$. Since $Q$ is a valuation ring then $N$ is prime. By way of contradiction suppose $\exists a, b\in N$ such that neither divides the other. We may assume that $Ra\cap Rb=0$. Let $A$ and $B$ be  maximal submodules of $Ra$ and $Rb$ respectively. We may replace $R$ by $R/(A+B)$ and assume that $Ra$ and $Rb$ are simple modules. Let $L$ and $P$ be their respective annihilators. Since $R_L$ is a valuation ring and $R_La\ne 0$, we have $R_Lb=0$. So, $L\ne P$. It follows that $\exists c\in L$ such that $(1-c)\in P$. Neither $c$ nor $(1-c)$ is regular. This contradicts that $R$ is almost clean.

Conversely, suppose that $N$ is prime and uniserial. Then, if $A$ is an ideal contained in $N$, $N/A$ is also uniseriel. So,  the ring of fractions of $R/A$ is a valuation ring: see \cite[p.218, between the definition of a torch ring and Theorem B]{Vam77}. \end{proof}

\bigskip
Following V\' amos \cite{Vam77},  we say that $R$ is a {\it torch ring} if
the following conditions are satisfied~:
\begin{enumerate}
\item $R$ is an arithmetical ring with at least two maximal ideals;
\item $R$ has a unique minimal prime ideal $N$ which is a
nonzero uniserial module.
\end{enumerate}

We follow T.S. Shores and R. Wiegand \cite{ShWi74}, by defining a {\it canonical form} for
an $R$-module $E$ to be a decomposition $E\cong R/I_1\oplus
R/I_2\oplus\dots\oplus R/I_n,$ where $I_1\subseteq
I_2\subseteq\dots\subseteq I_n\not= R,$ and by calling a ring $R$ a
{\it CF-ring} if every direct sum of finitely many cyclic modules has a
canonical form.

\begin{corollary}
\label{C:CF} Each CF-ring is almost clean.
\end{corollary}
\begin{proof}
By \cite[Theorem 3.12]{ShWi74} every CF-ring is
arithmetical and a finite product of indecomposable CF-rings. If $R$ is
indecomposable then $R$ is either a domain, or a local ring, or a torch ring.  By Corollary~\ref{C:fracValu} $R$ is almost clean.  \end{proof}

\medskip

By Proposition~\ref{P:stalk} there is some similarity between \cite[Theorem 2.4]{BuRap08} and the following theorem.

\begin{theorem} \label{T:hausd} 
Let $R$ be a ring. Consider the following conditions:
\begin{enumerate}
\item $R$ is almost clean;
\item $\mathrm{pSpec}\ R$ is totally disconnected 
and $\forall r\in R$,  $\forall x\in \mathrm{pSpec}\ R$, $\exists s_x\in\mathfrak{R}(R)$ such that either $r\equiv s_x$ modulo $A(x)$ or $(r-1)\equiv s_x$ modulo $A(x)$;
\item $\mathrm{pSpec}\ R$ is totally disconnected 
and for each $x\in \mathrm{pSpec}\ R$, $R/A(x)$ is almost clean. 
\end{enumerate} 
Then $(i)\Leftrightarrow(ii)\Rightarrow (iii)$ and the three conditions are equivalent if every principal ideal of $R$ is finitely presented.
\end{theorem}

\begin{proof} $(i)\Rightarrow (ii)$. Let $x$ and $y$ be two distinct points of $\mathrm{pSpec}\ R$. Let $P$ and $P'$ be two minimal prime ideals of $R$ such that $P\in x$ and $P'\in y$. There is no maximal ideal containing $P$ and $P'$. So, $P+P'=R$ and there exist $a\in P$ and $a'\in P'$ such that $a+a'=1$. We have $a=s+e$ where $s$ is  regular and $e$  idempotent. Since $s\notin P$ we get that $e\notin P$. It follows that for each $L\in x$, $(1-e)\in L$ and $e\notin L$. So, $x\subseteq D(e)$. We have $a'=-s+(1-e)$. In the same way we get $y\subseteq D(1-e)$. Therefore $x$ and $y$ belong to disjoint clopen neighbourhoods. Hence $\mathrm{pSpec}\ R$ is  totally disconnected. Now, let $r\in R$ and  $z\in\mathrm{pSpec}\ R$. We have $r=s+e$ where $s$ is  regular and $e$  idempotent. If $z\subseteq V(e)$ then $e\in A(z)$ and $r\equiv s$ modulo $A(z)$; and, if $z\subseteq V(1-e)$ then $(1-e)\in A(z)$ and $(r-1)\equiv s$ modulo $A(z)$.

$(ii)\Rightarrow (i)$. Let $a\in R$. Let $Q$ be the ring of fractions of $R$ and let $S=R\ltimes Q$ be the trivial extension of $R$ by $Q$. We consider the following polynomial equations in $S$: $E^2=E, E+X=(a,0)$ and $XY=(0,1)$. Let $x\in\mathrm{pSpec}\ R$. If $a\equiv s_x$ modulo $A(x)$ where $s_x$ is a regular element of $R$, then $E=(0,0),\ X=(s_x,0),\ Y=(0,1/s_x)$ is a solution of these polynomial equations modulo $A(x)S$; if $(a-1)\equiv s_x$ modulo $A(x)$, we take $E=(1,0)$. So, by Lemma~\ref{L:poly}, these equations have a solution in $S$: $E=(e,q),\ X=(s,u),\ Y=(t,v)$. From $E^2=E$ we deduce that $e^2=e$ and $(2e-1)q=0$. So, $q=0$ since $(2e-1)$ is a unit. From $E+X=(a,0)$ we deduce that $u=0$, and from $XY=(0,1)$ we deduce that $sv=1$. Hence $s$ is a regular element of $R$ and $a=e+s$. We conclude that $R$ is almost clean.

$(ii)\Rightarrow (iii).$  Clearly, if $r\in R$ then, $\forall x\in\mathrm{pSpec}\ R$ either $r$ or $(r-1)$ is regular modulo $A(x)$. So, $R/A(x)$ is almost clean  $\forall x\in\mathrm{pSpec}\ R$.

$(iii)\Rightarrow (ii).$ We assume that each principal ideal is finitely presented. Let $x\in\mathrm{pSpec}\ R$. It remains to show that any regular element $a$ modulo $A(x)$ is congruent to a regular element of $R$ modulo $A(x)$. Since $(0:a)$ is finitely generated and $A(x)$ is generated by idempotents, there exists an idempotent $e\in A(x)$ such that $(0:a)\subseteq Re$. Now, it is easy to check that $a(1-e)+e$ is a regular element. 
   \end{proof}

The following examples show that the conditions $(i)$ and $(iii)$ are not generally equivalent.
\begin{example}
\label{E:BurRap} \cite[Examples 2.2 and 2.9.(i)]{BuRap08} are non-almost clean arithmetical rings (the second is reduced) with almost clean stalks. These are defined in the following way: let $D$ be a principal ideal domain and $\forall n\in\mathbb{N}$, let $R_n$ be a  quotient of $D$ by a non-zero proper ideal $I_n$; let $R$ be the set of elements $r=(r_n)_{n\in\mathbb{N}}$ of $\Pi_{n\in\mathbb{N}}R_n$ which satisfy $\exists m_r\in\mathbb{N}$ and $\exists d_r\in D$ such that $\forall n\geq m_r$, $r_n=d_r+I_n$. We put $e_m=(\delta_{m,n})_{n\in\mathbb{N}},\ \forall m\in\mathbb{N}$. It is easy to check that the points of $\mathrm{pSpec}\ R$ are: $x_{\infty}=V(\bigoplus_{n\in\mathbb{N}}R_n)$ and $\forall n\in\mathbb{N}$, $x_n=V(1-e_n)$. Since $x_{\infty}\subseteq D(1-e_m)$ and $x_m\subseteq D(e_m)$, $\forall m\in\mathbb{N}$, $\mathrm{pSpec}\ R$ is totally disconnected. So, (by using Proposition~\ref{P:stalk}) these examples satisfy condition $(iii)$ of Theorem~\ref{T:hausd}.
\end{example}

The following example shows that the condition "each principal ideal is finitely presented" is not necessary if $R$ is almost clean.
\begin{example}
\label{E:vas} We consider \cite[Example 1.3b]{Vas76}. Let $R=\mathbb{Z}\oplus S$ where $S=(\mathbb{Z}/2\mathbb{Z})^{\mathbb{(N)}}$. The multiplication is defined by $(m,x)(n,y)=(mn,nx+my+xy)$, where $m,n\in\mathbb{Z}$ and $x,y\in S$. It is known that $R$ is a reduced arithmetical ring which is not semihereditary. Let $p\in\mathbb{Z}$ and $s\in S$. Then $(2p-1,s)=(2p-1,0)+(0,s)$ and  $(2p,s)=(2p-1,0)+(1,s)$. It is easy to check that $(2p-1,0)$ is regular and, $(0,s)$ and $(1,s)$ are idempotents.  Hence $R$ is almost clean. But $(0:(2,0))=0\times S$ which is not finitely generated.
\end{example}

\begin{proposition}
\label{P:fracidem} Let $R$ be a ring and $Q$ its ring of fractions. Assume that each prime ideal of $R$ contains only one minimal prime. Then:
\begin{enumerate}
\item $\mathrm{pSpec}\ Q$ and $\mathrm{pSpec}\ R$ are homeomorphic; 
\item each idempotent of $Q$ belongs to $R$;
\item if $Q$ is clean then $R$ is almost clean.
\end{enumerate} 
\end{proposition}
\begin{proof} $(i)$. If $\varphi:R\rightarrow Q$ is the natural map then $\mathrm{Min}\ R\subseteq\Phi_R=\mathrm{Im}\ {}^a\varphi$. It follows that each prime ideal of $Q$ contains only one minimal prime and $^b\varphi$ is bijective. Moreover, since $\mathrm{pSpec}\ Q$ and $\mathrm{pSpec}\ R$ are compact, $^b\varphi$ is a homeomorphism.

$(ii)$. Let $e$ an idempotent of $Q$. Then $\lambda_Q^{\rightarrow}(D(e))$ is a clopen subset of $\mathrm{pSpec}\ Q$, whence its image by $^b\varphi$ is a clopen subset of $\mathrm{pSpec}\ R$ and consequently it is of the form $\lambda_R^{\rightarrow}(D(e'))$ where $e'$ is an idempotent of $R$. But the inverse image of $D(e')\subseteq\mathrm{Spec}\ R$ by $^a\varphi$ is $D(e')\subseteq\mathrm{Spec}\ Q$. So, $e=e'\in R$. 

$(iii).$ Assume that $Q$ is clean. Let $r\in R$. Then $r=q+e$ where $q$ is a unit of $Q$ and $e$ an idempotent. Since $e\in R$, $q$ is a regular element of $R$. \end{proof}

\begin{corollary} \label{C:FracClean}
Let $R$ be an almost clean arithmetical ring and $Q$ its ring of fractions. If $Q$ is coherent then $Q$ is clean. In this case $Q$ is an elementary divisor ring.
\end{corollary}
\begin{proof} Recall that an arithmetical ring is coherent if and only if each principal ideal is finitely presented because the intersection of any two finitely generated ideals is finitely gene\-rated by \cite[Corollary 1.11]{ShWi74}. By Proposition~\ref{P:fracidem} we may assume that $\mathrm{pSpec}\ Q=\mathrm{pSpec}\ R$. This space is  totally disconnected by Theorem~\ref{T:hausd}. Let $x\in\mathrm{pSpec}\ Q$ and let $A(x)$ be the pure ideal of $R$ such that $x=V(A(x))$. Then $x=V(QA(x))$. If $s$ is a regular element of $R$ then $s+A(x)$ is a regular element of $R/A(x)$. So, $R/A(x)$ and $Q/QA(x)$ have the same ring of fractions which is a valuation ring by Corollary~\ref{C:fracValu}. Hence $Q/QA(x)$ is almost clean, $\forall x\in\mathrm{pSpec}\ Q$. By Theorem~\ref{T:hausd} $Q$ is almost clean too.  We conclude that $Q$ is clean since each regular element is a unit.

The last assertion is a consequence of \cite[Theorem I.1 and Corollary II.2]{Cou07}.
 \end{proof}

\medskip
We don't know if the assumption "$Q$ is coherent" can be omitted. The following example shows that the conclusion of the previous corollary doesn't hold if $R$ is not arithmetical, even if $R$ has a unique minimal prime ideal.
\begin{example}
Let $K$ be a field,  $D=K[x,y]_{(x,y)}$ where $x,y$ are indeterminates, $E=D/Dx\oplus D/Dy$ and $R$ the trivial extension of $D$ by $E$. Since $R$ contains a unique minimal prime ideal, $Q$ is indecomposable. We put $r=\binom{x\ \ 0}{0\ \ x}$ and $s=\binom{y\ \ 0}{0\ \ y}$. Clearly $r$ and $s$ are zerodivisors but $r+s$ is regular. It follows that $\frac{r}{r+s}$ and $\frac{s}{r+s}$ are two zerodivisors of $Q$ whose sum is $1$. So, $Q$ is not almost clean.
\end{example}

\section{Fractionally IF-rings}
\label{S:IF}
Let $\mathcal{P}$ be a ring property. We say that a ring $R$ is \textit{(finitely) fractionally} $\mathcal{P}$ if the classical ring of quotients $Q(R/A)$ of $R/A$  satisfies $\mathcal{P}$ for each (finitely generated) ideal $A$. In \cite{FaFa97} Facchini and Faith studied fractionally self FP-injective rings. They proved that these rings are arithmetical (\cite[Theorem 1]{FaFa97}) and they gave some examples (\cite[Theorem 6]{FaFa97}). In the first part of this section we investigate fractionally IF-rings and  partially answer  a question posed by Facchini and Faith in \cite[question Q1 p.301]{FaFa97}.

Some preliminary results are needed.
As in \cite{Mat85} a ring $R$ is said to be \textit{semicoherent} if $\mathrm{Hom}_R(E,F)$ is a submodule of a flat $R$-module for any pair of injective $R$-modules $E,\ F$. An $R$-module $E$ is {\it FP-injective} if 
$\hbox{Ext}_R^1 (F, E) = 0$  for any finitely presented $R$-module $F,$ and $R$ is {\it self FP-injective} if $R$ is 
FP-injective as $R$-module.  We recall that a module $E$ is FP-injective if and only if it is a pure submodule
of every overmodule. If each injective $R$-module is flat we say that $R$ is an \textit{IF-ring}. By \cite[Theorem 2]{Col75}, $R$ is an IF-ring if and only if it is coherent and self FP-injective.

\begin{proposition}
\label{P:FPinj} Let $R$ be a self FP-injective ring. Then $R$ is coherent if and only if it is semicoherent.
\end{proposition}
\begin{proof} If $R$ is coherent then $\mathrm{Hom}_R(E,F)$ is flat for any pair of injective modules $E,\ F$ by \cite[Theorem~XIII.6.4(b)]{FuSa01}; so, $R$ is semicoherent. Conversely, let $E$ be the injective hull of $R$. Since $R$ is a pure submodule of $E$, then, for each injective $R$-module $F$, the following sequence is exact:
\[0\rightarrow\mathrm{Hom}_R(F\otimes_RE/R,F)\rightarrow\mathrm{Hom}_R(F\otimes_RE,F)\rightarrow\mathrm{Hom}_R(F\otimes_RR,F)\rightarrow 0.\]
By using the natural isomorphisms $\mathrm{Hom}_R(F\otimes_RB,F)\cong\mathrm{Hom}_R(F,\mathrm{Hom}_R(B,F))$ and $F\cong\mathrm{Hom}_R(R,F)$ we get the following exact sequence:
\[0\rightarrow\mathrm{Hom}_R(F,\mathrm{Hom}_R(E/R,F))\rightarrow\mathrm{Hom}_R(F,\mathrm{Hom}_R(E,F))\rightarrow\mathrm{Hom}_R(F,F)\rightarrow 0.\]
So, the identity map on $F$ is the image of an element of $\mathrm{Hom}_R(F,\mathrm{Hom}_R(E,F))$.
Consequently the following sequence splits:
\[0\rightarrow\mathrm{Hom}_R(E/R,F)\rightarrow\mathrm{Hom}_R(E,F)\rightarrow F\rightarrow 0.\]
It follows that $F$ is a direct summand of a flat module. So, $R$ is an IF-ring.  \end{proof}

\begin{corollary}
\label{C:semicoh} Let $R$ be a ring. Assume that its ring of quotients $Q$ is self FP-injective. Then $R$ is semicoherent if and only if $Q$ is coherent.
\end{corollary}
\begin{proof} If $R$ is semicoherent, then so is $Q$ by \cite[Proposition 1.2]{Mat85}. From Proposition~\ref{P:FPinj} we deduce that $Q$ is coherent. Conversely, let $E$ and $F$ be injective $R$-modules. It is easy to check that the multiplication by a regular element of $R$ in $\mathrm{Hom}_R(E,F)$ is injective. So, $\mathrm{Hom}_R(E,F)$ is a submodule of the injective hull of $Q\otimes_R\mathrm{Hom}_R(E,F)$ which is flat over $Q$ and $R$ because $Q$ is an IF-ring.  \end{proof}

\begin{corollary}
\label{C:quotval} Let $R$ be a valuation ring and $A$ an ideal. Then $R/A$ is semicoherent if and only if $A$ is either prime or the inverse image of a proper principal ideal of $R_{A^{\sharp}}$ by the natural map $R\rightarrow R_{A^{\sharp}}$, where  $A^{\sharp}=\{r\in R\mid rA\subset A\}$.
\end{corollary}
\begin{proof} Assume that $A$ is not prime and let $A'=AR_{A^{\sharp}}$. Then $R_{A^{\sharp}}/A'$ is the ring of quotients of $R/A$. So, by \cite[Th\'eor\`eme 2.8]{Cou82}, $R_{A^{\sharp}}/A'$ is self FP-injective because each non-unit is a zero-divisor. By \cite[Corollary II.14]{Cou03} it is coherent if and only if $A'$ is principal. So, we conclude by Corollary~\ref{C:semicoh}.  \end{proof}

\bigskip

A valuation ring $R$ is called \textit{strongly discrete} if there is no non-zero idempotent prime ideal.  

\begin{corollary}
\label{C:StrDis} Let $R$ be a valuation ring. Then $R/A$ is semicoherent for each ideal $A$ if and only if $R$ is strongly discrete.
\end{corollary}
\begin{proof} Assume that $R$ is strongly discrete. Each ideal $A$ is of the form $A=aL$, where $L$ is a prime ideal and $a\in R$. Clearly $L=A^{\sharp}$. Then, $L^{2}\ne L$ implies that $AR_L$ is principal over $R_L$. Since $A$ is the inverse image of $AR_L$ by the natural map $R\rightarrow R_L$, $R/A$ is semicoherent by Corollary~\ref{C:quotval}.
 
 Conversely, let $L$ be non-zero prime ideal, let $A=aLR_L$, where $0\ne a\in R_L$ and let $A'$ be the inverse image of $A$ by the natural map $R\rightarrow R_L$. Clearly $L=(A')^{\sharp}$. Since $R/A'$ is semicoherent, $A$ is principal over $R_L$ by Corollary~\ref{C:quotval}. It follows that  $L$ is principal over $R_L$. So, $L\ne L^2$.  \end{proof}

\bigskip

Now, we can prove one of the main results of this section.
\begin{theorem}
\label{T:fracIF} Let $R$ be a fractionally IF-ring. Then, $R/A$ is  semicoherent for each ideal $A$ and $R_P$ is  a strongly discrete valuation ring for each maximal ideal $P$. 
\end{theorem}
\begin{proof}  By \cite[Theorem 1]{FaFa97} $R$ is arithmetical because it  is fractionally self FP-injective. Let $P$ be a maximal ideal and let $A$ be an ideal of $R_P$. If $B$ is the kernel of the following composition of natural maps $R\rightarrow R_P\rightarrow R_P/A$, then $Q(R_P/A)=Q(R/B)$ is an IF-ring. We conclude by Corollaries~\ref{C:quotval} and \ref{C:StrDis}.
 \end{proof} 

\medskip

 It is obvious that each von Neumann regular ring is  a fractionally IF-ring. Moreover:
\begin{proposition}
\label{P:Examp} Let $R$ be an arithmetical ring which is locally strongly discrete. Then $R$ is a fractionally IF-ring in the following cases:
\begin{enumerate}
\item $R$ is fractionally semilocal;
\item $R$ is semilocal;
\item $R$ is a Pr\"ufer domain of finite character, i.e. each non-zero element is contained in but a finite number of maximal ideals.
\end{enumerate} 
\end{proposition}
\begin{proof} $(i)$. We may assume that $R=Q(R)$. By \cite[Lemma 7]{FaFa97}, for each maximal ideal $P$, $R_P=Q(R_P)$. It follows that $R_P$ is self FP-injective. Moreover it is coherent by Corollary~\ref{C:StrDis} and Proposition~\ref{P:FPinj}. Since $\Pi_{P\in\mathrm{Max}\ R}R_P$ is a faithfully flat $R$-module and an IF-ring, we deduce that $R$ is IF too. 

$(ii)$ follows from $(i)$ by \cite[Lemma 5]{FaFa97}.

$(iii)$ follows from $(ii)$ since $R/A$ is semilocal for each  non-zero ideal $A$.  \end{proof}

\begin{question}
What are the locally strongly discrete Pr\"ufer domains which are fractionally IF?
\end{question}

%\medskip

The following example shows that  an arithmetical ring which is locally  Artinian  is not necessarily fractionally IF.
\begin{example}\label{E:nonQIF}
Let $K$ be a field, $V=K[X]/(X^2)$ where $X$ is an indeterminate and let $x$ be the image of $X$ in $V$. For each $p\in\mathbb{N}$ we put $R_{2p}=V$ and $R_{2p+1}=V/xV\cong K$. Let $S=\prod_{n\in\mathbb{N}}R_n$, $J=\bigoplus_{n\in\mathbb{N}}R_n$ and let $R$ be the unitary $V$-subalgebra of $S$ generated by $J$. For each $n\in\mathbb{N}$ we set $\mathbf{e}_n=(\delta_{n,p})_{p\in\mathbb{N}}$ and we denote by $\mathbf{1}$ the identity element of $R$. Let $P$ be a maximal ideal of $R$:
\begin{itemize}
\item either $J\subseteq P$; in this case $P=P_{\infty}=J+xR$ and $R_{P_{\infty}}=R/J\cong V$;
\item or $\exists n\in\mathbb{N}$ such that $\mathbf{e}_n\notin P$; in this case $P=P_n=R(\mathbf{1}-\mathbf{e}_n)+Rx\mathbf{e}_n$, $R_{P_n}=R/R(\mathbf{1}-\mathbf{e}_n)\cong V$ if $n$ is even and $R_{P_n}\cong K$ if $n$ is odd.
\end{itemize} 
Then, for each maximal ideal $P$, $R_P$ is an artinian valuation ring. Now it is easy to check that $(0:x\mathbf{1})=Rx\mathbf{1}+\bigoplus_{n\in\mathbb{N}}R\mathbf{e}_{2n+1}$. So, $R$ is not coherent.
\end{example}

\medskip

The following proposition is a short answer to another question posed by Facchini and Faith in \cite[question Q3 p.301]{FaFa97}.
\begin{proposition}
There exists a non-arithmetical zero-Krull-dimensional ring $R$ which is finitely fractionally self FP-injective.
\end{proposition}
\begin{proof} Let $V$ be the artinian valuation ring of Example~\ref{E:nonQIF}, $S=V^{\mathbb{N}}$ and $J=V^{(\mathbb{N})}$. Let $y=(y_n)_{n\in\mathbb{N}},\ z=(z_n)_{n\in\mathbb{N}}\in S$ such that, $\forall p\in\mathbb{N},\ y_{2p}=z_{2p+1}=x$ and $y_{2p+1}=z_{2p}=0$, and let $R$ be the unitary $V$-subalgebra of $S$ generated by $y,z$ and $J$. The idempotents $(\mathbf{e}_n)_{n\in\mathbb{N}}$ are defined as in Example~\ref{E:nonQIF}.
Let $P\in\mathrm{Max}\ R$:
\begin{itemize}
\item either $J\subseteq P$; in this case $P=P_{\infty}=J+yR+zR$ and $R_{P_{\infty}}=R/J\cong K[Y,Z]/(Y,Z)^2$;
\item or $\exists n\in\mathbb{N}$ such that $\mathbf{e}_n\notin P$; in this case $P=P_n=R(\mathbf{1}-\mathbf{e}_n)+Rx\mathbf{e}_n$, $R_{P_n}=R/R(\mathbf{1}-\mathbf{e}_n)\cong V$.
\end{itemize} 
Clearly, $R_{P_{\infty}}$ is not a valuation ring. So, $R$ is not arithmetical. First, we show that $R$ is a pure submodule of $S$. It is sufficient to prove that $R_P$ is a pure submodule of $S_P$ for each maximal ideal $P$. It is obvious that $R_{P_n}\cong  S_{P_n}\cong\mathbf{e}_nS\cong V$. It remains to be shown that $R_{P_{\infty}}$ is a pure submodule of $S_{P_{\infty}}\cong S/J$. We consider the following equations:
\[\forall i,\ 1\leq i\leq p, \sum_{1\leq j\leq m}r_{i,j}x_j\equiv s_i\ \mathrm{modulo}\ J,\]
where $r_{i,j},\ s_i\in R,\ \forall i,\ 1\leq i\leq p,\ \forall j,\ 1\leq j\leq m$. When these equations have a solution in $S$, we must prove they have  a solution in $R$ too. This can be done by using the basis $\{\mathbf{1},y,z,\mathbf{e}_n,x\mathbf{e}_n\mid n\in\mathbb{N}\}$ of $R$ over $K$. 
Consequently $R$ is pure in $S$. Now, let $A$ be a finitely generated ideal of $R$. Then $R/A$ is a pure submodule of $S/SA$. We have $S/SA\cong\prod_{n\in\mathbb{N}}(R/A)_{P_n}$. For each $n\in\mathbb{N}$, $(R/A)_{P_n}$ is self injective, whence it is an injective $(R/A)$-module. We deduce that $S/SA$ is  injective over $R/A$. Hence $R/A$ is self FP-injective.  \end{proof}

\medskip

Finally, for the second question posed by Facchini and Faith in \cite[question Q2 p.301]{FaFa97}, we shall prove Theorem~\ref{T:ffIF}. The following lemma is needed.
\begin{lemma}
Let $R$ be a clean ring such that $R=Q(R)$ and $(0:a)$ is finitely generated for each $a\in R$. Then, for each maximal ideal $P$, $R_P=Q(R_P)$.
\end{lemma} \label{L:CleFrac}
\begin{proof}
Let $P$ be a maximal ideal of $R$. By way of contradiction, suppose that $R_P$ contains a regular element which is not a unit. So,  $\exists a\in P$ such that $(0_P:a)=0_P$ (since $R$ is clean, $R_P=R/0_P$ by \cite[Proposition III.1]{Cou07}). It follows that $(0:a)\subseteq 0_P$. Since $(0:a)$ is finitely generated and $0_P$ is generated by idempotents, there exists an idempotent $e\in 0_P$ such that $(0:a)\subseteq Re$. Now, it is easy to check that $a(1-e)+e$ is a regular element contained in $P$. This contradicts that $R=Q(R)$.
\end{proof}

\begin{theorem} \label{T:ffIF} The following assertions hold:
\begin{enumerate}
\item let $R$ be an almost clean coherent arithmetical ring. Assume that $R/A(x)$ is either torch or local or a domain $\forall x\in\mathrm{pSpec}\ R$. Then $R$ is finitely fractionally IF;
\item each clean coherent arithmetical ring is finitely fractionally IF;
\item each semihereditary ring is finitely fractionally IF;
\item let $R$ be a zero-Krull-dimensional ring or a one-Krull-dimensional domain. Then $R$ is finitely fractionally IF if and only if $R$ is coherent and arithmetical.
\end{enumerate} 
\end{theorem}
\begin{proof}
 
$(i)$. Let $A$ be a finitely generated ideal  of an almost clean coherent arithmetical ring $R$ and let $N$ be the nilradical of $R$. Since each prime ideal contains only one minimal prime, by \cite[Lemme IV.2]{Cou07} $D((N:A))$ is the inverse image by $\lambda_R$ of an open subset $U$ of $\mathrm{pSpec}\ R$. For each $x\in U$ there exists an idempotent $e_x\in (N:A)$ such that $x\subseteq D(e_x)\subseteq D((N:A))$. Hence $(N:A)=\Sigma_{x\in U}(Re_x+N)$. Since $R$ is coherent $(0:A)$ is finitely generated. It follows that there exists an idempotent $e\in (N:A)$ such that $(0:A)\subseteq (Re+N)$. We have $R/A\cong (R(1-e)/A(1-e))\times(Re/Ae)$. Since $(0:_{R(1-e)}A(1-e))\subseteq N(1-e)$, $(R(1-e)/A(1-e))$ is IF by \cite[Proposition II.15]{Cou03}. On the other hand, $Re$ is an almost coherent arithmetical ring and $Ae$ is a finitely generated ideal contained in $Ne$. Let $T=(Re/Ae)$ and let $\phi:R\rightarrow T$ be the natural epimorphism. Then $\mathrm{pSpec}\ T$ is totally disconnected because it is homeomorphic to $\mathrm{pSpec}\ (Re)$.  Let $x\in\mathrm{pSpec}\ T$. Then $T/A(x)$ is the quotient of $R/A(^b\phi(x))$ modulo an ideal contained in the minimal prime  of $R/A(^b\phi(x))$. By Corollary~\ref{C:fracValu} $T/A(x)$ is almost clean. We deduce that $T$ is coherent and almost clean by  Theorem~\ref{T:hausd}. By Corollary~\ref{C:FracClean} $T'=Q(T)$ is clean. By Lemma~\ref{L:CleFrac} $T'_P=Q(T'_P)$ for each maximal ideal $P$ of $T'$. We deduce that $T'_P$ is IF because it is a valuation ring. So, since $T'$ is locally IF, it is IF too. Hence $Q(R/A)$ is IF. 
 
$(ii)$ and $(iii)$. If $R$ is either clean, coherent and arithmetical or semihereditary, then $R$ satisfies the conditions of $(i)$. Hence $R$ is finitely fractionally IF.

$(iv)$. Observe that $Q(R/A)=R/A$ for each non-zero proper ideal $A$.

First assume that $R$ is arithmetical and coherent. We deduce that $R$ is finitely fractionally IF from $(ii)$ and $(iii)$.

Conversely, let $P$ be a maximal ideal of $R$ and  $A$  a finitely generated ideal of $R_P$. There exists a finitely generated ideal $B$ of $R$ such that $A=B_P$. So, $R_P/A\cong(R/B)_P$. Since $R/B$ is IF, so is $R_P/A$ by \cite[Proposition 1.2]{Cou82}. Hence we may assume that $R$ is local and we must prove that $R$ is a valuation ring. If not, there exist $a,b\in R$ such that $a\notin Rb$ and $b\notin Ra$. The coherence of $R/(ab)$ implies that $Ra\cap Rb$ is finitely generated. It follows that $R/(Ra\cap Rb)$ is IF. We may assume that $Ra\cap Rb=0$. By \cite[Corollary 2.5]{Ja73} $A=(0:(0:A))$ for each finitely generated ideal  $A$. We deduce that $0=Ra\cap Rb=(0:(0:a)+(0:b)).$ Then $(0:a)+(0:b)$ is a faithful finitely generated proper ideal. By \cite[Corollary 2.5]{Ja73} this is not possible. Hence $R$ is a valuation ring. \end{proof}

\bigskip
If $R$ is fractionally self FP-injective, then, by \cite[Theorem III.1]{Cou03} $\mathrm{Min}\ R/A$ is compact for each proper ideal $A$. The following example shows that this is not true if $R$ is finitely fractionally IF, even if $R$ is a coherent clean arithmetical ring.

\begin{example} Let $D$ be a valuation domain. Assume that its maximal ideal $P'$ is the only non-zero prime and it is not finitely generated. Let $0\ne d\in P'$. We put $V=D/dD$, $P=P'/dD$ and $R=V^{\mathbb{N}}$. It is easy to check that $R$ is clean, B\'ezout and coherent. So, by (ii) of Theorem~\ref{T:ffIF} $R$ is finitely fractionally IF. Since $P$ is not finitely generated, $\forall n\in\mathbb{N}$, $\exists b_n\in P$ such that $b_n^n\ne 0$. We set $b=(b_n)_{n\in\mathbb{N}}$. Let $N$ be the nilradical of $R$. If there exists $c=(c_n)_{n\in\mathbb{N}}\in R$ such that $(b-bcb)\in N$, then $\exists m\in\mathbb{N}$ such that $b^m(1-cb)^m=0$. If $n\in\mathbb{N},\ n\geq m$ we get that $b_n^n(1-c_nb_n)^n=0$. Clearly there is a contradiction. So, $R/N$ is not Von Neumann regular. Now, let $c\in R$ such that $(N:c)=N$. We shall prove that $c$ is  a unit. By way of contradiction, suppose  $\exists k\in\mathbb{N}$  such that $c_k\in P$. We put $e_k=(\delta_{k,n})_{n\in\mathbb{N}}$. Then $ce_k=c_ke_k\in N$. It follows that $e_k\in N$, which is absurd. So, $\forall k\in\mathbb{N},\ c_k\notin P$. Therefore $c$ is a unit and $R/N$ is equal to its quotient ring. By \cite[Theorem 5]{End61} a reduced arithmetic ring $S$ is semihereditary if and only if $Q(S)$ is von Neumann regular. By \cite[Proposition 10]{Que71} a reduced arithmetic ring $S$ is semihereditary if and only if $\mathrm{Min}\ S$ is compact. Consequently $\mathrm{Min}\ R/N$ is not compact . Hence $\mathrm{Min}\ R$ is not compact too. (If $P'$ is finitely generated by $p$ and  $R=\prod_{n\in\mathbb{N}}D/p^{n+1}D$, then $R$ is clean, arithmetical, coherent and finitely fractionally IF, but $\mathrm{Min}\ R$ is not compact. We do the same proof by taking  $b_n=p+p^{n+1}D,\ \forall n\in\mathbb{N}$.) 
\end{example}

%*****************************************************%
% Please choose one of the following options:         %
%                                                     %
% VERSION A is for BIBTeX application.                %
% You have to add the names of your databases:        %
%                                                     %
%\bibliography{database}
%                                                     %
% IF you prefer VERSION B for the LaTeX standard      %
% bibliography environment, please use the same style %
% as produced by commutativealgebra.bst (see the      %
% examples in the instructions for authors).          %
%                                                     %
%\begin{thebibliography}{99}                           %
%\bibitem{Lange:2007}                                  %
%T.~Lange and I.~E. Shaparlinski, Distribution of      %
%some sequences of points on elliptic curves,          %
%\emph{J. Math. Cryptol.} \textit{1} (2007), 1--11.    %
%\end{thebibliography}                                 %
%                                                     %
%*****************************************************%

\end{document}